\documentclass[10pt]{article}
\usepackage{amsmath,amssymb,amsthm}
\usepackage{epsfig}
\usepackage{color}

\title{New bounds on the signed domination numbers of graphs}

\author {
S.M. Hosseini Moghaddam\\
Islamic Azad University\\
Behshahr Branch, Behshahr, IRI\\
{\tt sm.hosseini1980@yahoo.com}\vspace{3mm}\\
Abdollah Khodkar\thanks{corresponding author} \\
Department of Mathematics\\
University of West Georgia\\
Carrollton, GA 30118, USA\\
{\tt akhodkar@westga.edu}\vspace{3mm}\\
Babak Samadi\\
Department of Mathematics\\
Arak University\\
Arak, IRI\\
{\tt samadibabak62@gmail.com}\vspace{3mm}
}

\date{}

 \newtheorem{theorem}{Theorem}[section]
\newtheorem{corollary}[theorem]{Corollary}
\newtheorem{lemma}[theorem]{Lemma}

\theoremstyle{definition}

\begin{document}

\maketitle

\begin{abstract}
\noindent A signed dominating function of a graph $G$ with vertex set $V$
is a function $f:V\rightarrow \{-1,1\}$
such that for every vertex $v$ in $V$ the sum of the values of $f$ at $v$ and at every vertex
$u$ adjacent to $v$ is at least 1. The weight of $f$ is the sum of the values of $f$
at every vertex of $V$. The signed domination number of
$G$ is the minimum weight of a signed dominating function of $G$.
In this paper, we study the signed domination numbers of graphs and present new sharp
lower and upper bounds for this parameter. As an example, we prove that the
signed domination number of a tree of order $n$ with $\ell$ leaves and $s$ support vertices is at least
‎$\dfrac{n+4+2(\ell-s)}{3}‎‎$. \vspace{3mm}
\end{abstract}

\section{Introduction}
Throughout this paper, let $G$ be a finite connected graph with vertex set $V=V(G)$,
edge set $E=E(G)$, minimum degree $\delta=\delta(G)$ and maximum degree $\Delta=\Delta(G)$.
 We use \cite{we} for terminology and notation which are not defined here. For any vertex
 $v\in V$, $N(v)=\{u\in G\mid uv\in E(G)\}$ denotes the {\em open neighborhood} of $v$ in
 $G$, and $N[v]=N(v)\cup \{v\}$ denotes its {\em closed neighborhood}. A set  $S\subseteq V$
 is a {\em dominating set} in $G$ if each vertex in $V\setminus S$ is adjacent to at least one
 vertex in $S$. The {\em domination number $\gamma(G)$} is the minimum cardinality of a
 dominating set in $G$. A subset $B\subseteq V(G)$ is a {\em packing} in $G$ if for
 every distinct vertices $u,v‎\in B‎$, $N[u]‎‎\cap N[v]‎=‎\emptyset‎$. The {\em packing
 number} $\rho(G)$ is the maximum cardinality of a packing in $G$.\\
 In \cite{hh}, Harary and Haynes introduced the concept of tuple domination as a
 generalization of domination in graphs. Let $1\leq k\leq \delta(G)+1$.
 A set $D\subseteq V$ is a {\em $k$-tuple dominating set} in $G$ if $|N[v]\cap D|\geq k$,
 for all $v\in V(G)$. The {\em $k$-tuple domination number}, denoted by $\gamma_{\times k}(G)$,
 is the minimum cardinality of a $k$-tuple dominating set. In fact, the authors showed
 that every graph $G$ with $\delta\geq k-1$ has a $k$-tuple dominating set
 and hence a $k$-tuple domination number. It is easy to see that
 $\gamma_{\times1}(G)=‎\gamma(G)‎$. This concept has been studied by several
 authors including \cite{gghr,msh}.

 \noindent Gallant et al. \cite{gghr} introduced the concept of limited packing
 in graphs and exhibited some real-world applications of it to network security,
 market saturation and codes. A set of vertices $B\subseteq V$ is called a
 {\em $k$-limited packing set} in $G$ if
$|N[v]\cap B|\leq k$ for all $v\in V$, where $k\geq 1$. The {\em $k$-limited
packing number}, $L_{k}(G)$, is the largest number of vertices in a $k$-limited
packing set. When $k=1$ we have $L‎_{1}(G)‎=‎\rho(G)‎$.

\noindent Let $S\subseteq V$. For a real-valued function $f:V\rightarrow R$ we define
$f(S)=\sum_{v\in S }f(v)$. Also, $f(V)$ is the {\em weight} of $f$.
A {\em signed dominating function}, abbreviated SDF, of $G$ is defined in
\cite{dhhs} as a function $f:V\rightarrow \{-1,1\}$ such that $f(N[v])\geq1$,
for every $v \in V$. The {\em signed domination number}, abbreviated  SDN, of
 $G$ is $\gamma_{s}(G)= \min\{ f(V) \mid f  \mbox { is  a SDF  of } G \}$.
 This concept was defined in \cite{dhhs} and has been studied by several
 authors including \cite{cs,f,hw,hs}.

\noindent In this paper, we continue the study of the concept of the signed domination numbers of
graphs. The authors noted that most of the existing bounds on $\gamma_{s}(G)$ are lower bounds except those
that are related to regular graphs; for more information the reader can consult \cite{f}.
In Section 2, we prove that
$\gamma_{s}(G)\leq  n-2‎\lfloor‎\dfrac{2‎\rho(G)+‎\delta(G)-2‎‎}{2}‎‎\rfloor‎‎$,
for a graph $G$ of order $n$ with $‎\delta(G)‎‎\geq2‎$.
In Section 3, we find some new sharp lower bounds
on $\gamma_s(G)$ for a general graph $G$. The lower bound given in Part (i) of
Theorem \ref{LB2}
can also be found in \cite{hw} with a much longer proof than the one presented here.
We also prove that $\gamma_{s}(T)‎‎\geq ‎\dfrac{n+4+2(\ell-s)}{3}‎‎$,
for a tree of order $n$ with $\ell$ leaves and $s$ support vertices.
Furthermore we show that this bound is sharp.

\section{An upper bound}
We bound $\gamma_{s}(G)$ from above in terms of order, minimum degree
and packing number of $G$ using the concept of limited packing.

\begin{theorem} \label{UB1}
Let $G$ be a graph of order $n$ with $\delta\geq2$. Then
$$\gamma_{s}(G)\leq  n-2‎\lfloor‎\dfrac{2‎\rho(G)+‎\delta-2‎‎}{2}‎‎\rfloor‎‎$$
and this bound is sharp.
\end{theorem}

\begin{proof}Let $B$ be a $‎\lfloor‎‎\frac{‎\delta‎}{2}‎\rfloor‎‎$-limited
packing in $G$. Define $f:V\rightarrow \{-1,1\}$ by
$$ f(v)=\left \{
\begin{array}{lll}
-1 & \mbox{if} & v\in B \\
 1 & \mbox{if} & v\in V‎\setminus‎ B.
\end{array}
\right.$$
 Since $B$ is a $\lfloor\frac{\delta}{2}\rfloor$-limited packing in $G$,
 $|N[v]\cap (V‎\setminus‎ B)|\geq \deg(v)-\lfloor\frac{\delta}{2} \rfloor+1$.
 Therefore, for every vertex $v$ in $V$,
 $$f(N[v])=|N[v]\cap (V‎\setminus‎ B)|-|N[v]\cap B|\geq \deg(v)-
 \lfloor\frac{ \delta}{2} \rfloor+1-\lfloor\frac{\delta}{2}\rfloor\geq1.$$
 Therefore $f$ is a SDF of
 $G$ with weight $n-2|B|.$ So, by the definition of
 ${\lfloor\frac{\delta}{2}\rfloor}$-limited packing number,
\begin{equation}\label{EQ1}
\gamma_{s}(G)‎\leq n-2L_{\lfloor\frac{\delta}{2}\rfloor}(G).
\end{equation}
\noindent  We now claim that $B\neq V$. If $B=V$ and $u\in V$ such that $\deg(u)=\Delta$,
then $\Delta+1=|N[u]\cap B|\leq \lfloor\frac{\delta}{2}\rfloor$, a contradiction.

Now let $u\in V\setminus B$. It is easy to check that
$|N[v]\cap(B \cup \{u\})|\leq\lfloor\frac{\delta}{2}\rfloor+1$, for all $v\in V$.
Therefore $B\cup \{u\}$ is a $\lfloor \frac{\delta}{2}\rfloor+1$-limited packing
set in $G$. Hence,
$L_{\lfloor\frac{\delta}{2}\rfloor+1}(G) \geq |B\cup \{u\}|= L_{\lfloor \frac{\delta}{2}\rfloor}(G)+1$.
Repeating these inequalities, we obtain
$L_{\lfloor\frac{\delta}{2}\rfloor}(G)\geq L_{\lfloor\frac{\delta}{2}\rfloor-1}(G)+1\geq
\ldots \geq L_{1}(G)+\lfloor\frac{\delta}{2}\rfloor-1$, and since $L_{1}(G)=\rho(G)$, we conclude
\begin{equation}\label{EQ2}
L_{\lfloor\frac{\delta}{2}\rfloor}(G)\geq\rho(G)+\lfloor\frac{\delta}{2}\rfloor-1
\end{equation}
The upper bound now follows by Inequalities (\ref{EQ1}) and (\ref{EQ2}). Moreover, The bound is
sharp for the complete graph of order $n\geq 3$.
\end{proof}

\section{Lower bounds}
For convenience, for the rest of the paper we make use of the following notations.
Let $G$ be a graph and $f:V(G)\longrightarrow\{-1,1\}$ be a SDF of $G$.
Define $V^{+}=\{v\in V \mid f(v)=1 \}$ and $V^{-}=\{v\in V \mid f(v)=-1 \}$. Let $G^{+}=G[V^{+}]$
and $G^{-}=G[V^{-}]$ be the subgraphs of $G$ induced by $V^{+}$ and $V^{-}$, respectively.
We also let $E^{+}=|E(G^{+})|$ and $E^{-}=|E(G^{-})|$.
We consider $[V^{+},V^{-}]$ as the set of edges having one end point in
$V^{+}$ and the other in $V^{-}$, $V_{o}=\{v‎\in V‎ \mid \deg(v) \mbox{ is  odd} \}$
 and $V_{e}=\{v‎\in V ‎ \mid \deg(v)  \mbox{ is  even} \}$. Also
 $V_{o}^{+}=V_{o}\cap V^{+}$, $V_{o}^{-}=V_{o}\cap V^{-}$, $V_{e}^{+}=V_{e}\cap V^{+}$
 and $V_{e}^{-}=V_{e}\cap V^{-}$. Finally, $\deg_{G^{+}}(v)=|N(v)\cap V^{+}|$
 and $\deg_{G^{-}}(v)=|N(v)\cap V^{-}|$.
For a graph $G$, let $O=\{ v \in V \mid \deg(v)=0 \}$, $L=\{ v\in V \mid \deg(v)=1 \}$,
$S=\{v\in V \mid N(v)\cap V_{2} \neq \emptyset \}$, $C(G)=V\setminus(O\cup L\cup S )$
and $\delta^{*}=\min\{\deg(v) \mid v\in C(G) \}$. Obviously, if $C(G)= \emptyset$,
then $\gamma_{s}(G)=n$. Therefore, in the following discussions we assume,
without loss of generality, that $C(G)\neq \emptyset$.
Thus, $\delta^{*}\geq \max\{2,\delta\}$.

\begin{lemma}\label{Lem.3.1}
The following statements hold.
\begin{itemize}
\item[(i)] $(\lceil \dfrac{\delta^{*}}{2}\rceil+1)|V^{-}|
\leq |[V^{+},V^{-}]|\leq \lfloor\dfrac{\Delta}{2}\rfloor(|V^{+}\setminus L|)$,

\item[(ii)] $|V_{o}|+2|V^{-}|\leq2|E^{+}|-2|E^{-}|$.
\end{itemize}
\end{lemma}

\begin{proof}
(i) \quad Let $v \in V^{-}$. Since $f(N[v]) \geq1$ and $v \in C(G)$,
we have $\deg_{G^{+}}(v)\geq \lceil\dfrac{\deg(v)}{2}\rceil+1\geq \lceil\dfrac{\delta^{*}}{2}\rceil+1$.
Therefore, $|[V^{+}, V^{-}]| \geq (\lceil\dfrac{\delta ^{*}}{2}\rceil+1)|V^{-}|$.
On the other hand, all leaves and support vertices belong to $V‎^{+}‎$.
Now let $v‎\in‎V‎^{+}‎\setminus L$.
Then $\deg_{G^{-}}(v)\leq\lfloor\dfrac{\deg(v)}{2}\rfloor\leq \lfloor\dfrac{\Delta}{2}\rfloor$.
Therefore, $|[V^{+},V^{-}]|\leq \lfloor\dfrac{ \Delta}{2}\rfloor(| V^{+}\setminus L|)$.\vspace{3mm}

\noindent (ii) \quad We first derive a lower bound for $|[V^{+},V^{-}]| $. Let $v\in V^{-}$.
Since $f(N[v])‎\geq1‎$, we observe that $\deg_{G^{+}}(v)‎\geq \deg_{G^{-}}(v)+2‎$ and
$\deg_{G^{+}}(v)‎\geq \deg_{G^{-}}(v)+3$ when $\deg(v)$ is odd. This leads to
\begin{equation}\label{EQ5}
\begin{array}{lcl}
|[V^{+},V^{-}]|&=&\sum_{v\in V_o^-}\deg_{G^+}(v)+\sum_{v\in V_e^-}\deg_{G^+}(v)\\
&‎\geq‎ &\sum_{v\in V_o^-}(\deg_{G^-}(v)+3)+\sum_{v\in V_e^-}(\deg_{G^-}(v)+2)\\
&=&3|V_o^-|+2|V_e^-|+\sum_{v\in V^-}\deg_{G^-}(v)\\
&=&2|V^-|+2|E^-|+|V_o^-|.
\end{array}
\end{equation}
Now let $v \in V^{+}$. Since $f(N[v])‎\geq1‎$,
we observe that $\deg_{G^{+}}(v)‎\geq \deg_{G^{-}}(v)‎$
and $\deg_{G^{+}}(v)‎\geq \deg_{G^{-}}(v)‎+1$ when $\deg(v)$ is odd.
It follows that
\begin{equation}\label{EQ6}
\begin{array}{lcl}
|[V^{+},V^{-}]|&=&\sum_{v\in V_o^+}\deg_{G^-}(v)+\sum_{v\in V_e^+}\deg_{G^-}(v)\\
&‎\leq ‎&\sum_{v\in V_o^+}(\deg_{G^+}(v)-1)+\sum_{v\in V_e^+}(\deg_{G^+}(v))\\
&=&\sum_{v\in V^+}\deg_{G^+}(v)-|V_o^+|=2|E^+|-|V_o^+|.
\end{array}
\end{equation}
Together inequalities (\ref{EQ5}) and (\ref{EQ6}) imply the desired inequality.
\end{proof}

We are now in a position to present the following lower bounds.

\begin{theorem}\label{LB1}
Let $G$ be a  graph of order $n$, size $m$, maximum degree $\Delta$ and $\ell$ leaves.
Let $V_{o}=\{v‎\in V‎ \mid \deg(v) \mbox{ is  odd} \}$.
Then
\begin{itemize}
\item [(i)] $\gamma_{s}(G)\geq \dfrac{(\lceil\dfrac{\delta^{*}}{2}\rceil-
            \lfloor\dfrac{\Delta}{2}\rfloor+1)n+2‎\lfloor‎‎\dfrac{‎\Delta‎}{2}‎\rfloor \ell‎‎}
            {\lceil\dfrac{\delta^{*}}{2}\rceil+\lfloor\dfrac{\Delta}{2}\rfloor+1}$,

\item[(ii)] $\gamma_{s}(G)\geq \dfrac{(\lceil\dfrac{3\delta^{*}}{2}\rceil-
\lfloor\dfrac{3\Delta}{2}\rfloor+3)n+2(‎\lfloor‎\dfrac{‎\Delta‎}{2}‎‎\rfloor‎‎ \ell+|V_o|)}
{\lceil\dfrac{3\delta^{*}}{2}\rceil+\lfloor\dfrac{3\Delta}{2}\rfloor+3}$.
\end{itemize}
Furthermore these bounds are sharp.
\end{theorem}

\begin{proof}
(i) This is a straightforward result by Part (i) of Lemma \ref{Lem.3.1},
$|V^{+}|=\dfrac{n+\gamma_{s}(G)}{2}$ and $|V^{-}|=\dfrac{n-\gamma_{s}(G)}{2}$.\vspace{3mm}

\noindent (ii) We have
\begin{equation}\label{EQ7}
\begin{array}{lcl}
2|E^+|&=& \sum_{v\in V^+}\deg_{G^+}(v)=\sum_{v\in V^+}\deg(v)-\sum_{v\in V^+}\deg_{G^-}(v)\\
&‎\leq‎ &‎\Delta‎|V^+|-|[V^{+},V^{-}]|‎\leq \Delta‎|V^+|-(‎\lceil‎‎\dfrac{‎\delta^*‎}{2}‎\rceil‎‎+1)|V^-|‎.
\end{array}
\end{equation}
and
\begin{equation}\label{EQ8}
\begin{array}{lcl}
2|E^-|&=& \sum_{v\in V^-}\deg_{G^-}(v)=\sum_{v\in V^-}\deg(v)-\sum_{v\in V^-}\deg_{G^+}(v)\\
&‎\geq‎‎ &‎\delta‎^*|V^-|-|[V^{+},V^{-}]|‎\geq ‎\delta^*|V^-|-‎\lfloor‎\dfrac{‎\Delta‎}{2}‎‎\rfloor(|V^+|-\ell)‎‎‎‎.
\end{array}
\end{equation}
Part (ii) of Lemma \ref{Lem.3.1} and Inequalities (\ref{EQ7}) and (\ref{EQ8}) imply the desired lower bound.

\noindent The bounds are sharp for the complete graph $K_n$.
\end{proof}
The lower bound given in Part (i) of Theorem \ref{LB1} was first found by Haas and
Wexler \cite{hw} for a graph $G$ with
$‎\delta(G)‎\geq 2‎‎$ using a longer proof.
The lower bound given in Part (i) of Theorem \ref{LB1} is an improvement of the lower bound
found in \cite{hw} when $\delta(G)=1$.

As an application of the concepts of limited packing and
tuple domination we give a sharp lower bound on $\gamma_{s}(G)$
in terms of the order of $G$, $‎\delta(G)‎$, $‎\Delta(G)‎$
and domination number $\gamma(G)‎‎$.

\begin{theorem}\label{LB2}
For any graph $G$ of order $n$, minimum degree $‎\delta‎$ and maximum degree $‎\Delta‎$,
$$\gamma_{s}(G)‎\geq -n+2\max \{‎\lceil‎‎\dfrac{‎\Delta‎+2}{2}‎\rceil,
‎\lceil‎\dfrac{‎\delta+2‎\gamma(G)‎‎}{2}‎‎\rceil‎‎‎‎‎\}$$
and this bound is sharp.
\end{theorem}

\begin{proof}
We first prove the following claims.

\noindent {\bf ِClaim 1.} $\gamma_{s}(G)‎\geq -n+2‎\lceil‎‎\frac{‎\Delta‎+2}{2}‎\rceil‎‎$.

\noindent Let $f:V\rightarrow \{-1,1\}$ be a SDF of $G$ with
weight $f(V(G))=\gamma_{s}(G)$. Since $f(N[v])‎\geq1‎$, it follows that
$|N[v]‎\cap V^-‎|‎\leq ‎\lfloor‎\frac{‎\Delta‎}{2}‎‎\rfloor‎‎‎$ for every vertex $v\in V(G)$.
Therefore $V^-$ is a $‎\lfloor‎\frac{‎\Delta‎}{2}‎‎\rfloor‎‎‎$-limited packing set in $G$. Thus
\begin{equation}\label{EQ9}
‎(n-\gamma_{s}(G)‎)‎/2=|V^-|‎\leq L‎_{‎\lfloor‎\frac{‎\Delta‎}{2}‎‎‎\rfloor‎}(G).‎‎
\end{equation}
On the other hand, similar to the proof of Theorem \ref{UB1}, we have
$$L‎_{‎\lfloor‎\frac{‎\Delta‎}{2}‎‎‎\rfloor‎}(G)‎\leq L‎_{‎\lfloor‎\frac{‎\Delta‎}{2}‎‎‎\rfloor‎+1}(G)‎-1‎\leq \ldots\leq L‎_{‎\Delta‎+1}(G)‎‎‎-‎\lceil‎‎\frac{‎\Delta‎}{2}‎\rceil‎‎-1=n‎‎‎-‎\lceil‎‎\frac{‎\Delta‎}{2}‎\rceil‎‎-1.$$
Now Inequality (\ref{EQ9}) implies $2(‎\lceil‎\frac{‎\Delta‎}{2}‎‎\rceil+1‎‎)-n‎\leq \gamma_{s}(G)‎‎$, as desired.\vspace{3mm}

\noindent {\bf Claim 2.} $\gamma_{s}(G)‎\geq -n+2‎\lceil‎‎\dfrac{‎\delta+2‎\gamma(G)‎‎}{2}‎\rceil$.

\noindent Since $‎‎f(N[v])‎\geq1‎$, it follows that
$|N[v]‎\cap V^+‎|‎\geq ‎\lceil‎‎\dfrac{‎\delta‎}{2}‎‎‎\rceil‎‎‎+1$,
for every vertex $v\in V(G)$. Therefore $V‎^{+}‎$ is a
$(\lceil‎‎\frac{‎\delta‎}{2}‎‎‎\rceil‎‎‎+1)$-tuple dominating set in $G$. Thus
\begin{equation}\label{EQ10}
(n+\gamma_{s}(G))/2=|V^+|‎\geq ‎\gamma‎_{‎\times‎(‎\lceil‎‎\frac{‎\delta‎}{2}‎\rceil+1‎‎)}(G).‎‎‎
\end{equation}
Now let $D$ be a $(\lceil‎‎\frac{‎\delta‎}{2}‎‎‎\rceil‎‎‎+1)$-tuple dominating set in $G$.
Then $|N[v]‎\cap D‎|‎\geq \lceil‎‎\frac{‎\delta‎}{2}‎‎‎\rceil‎‎‎+1‎$, for every vertex $v‎\in V(G)‎$.
Let $u‎\in D‎$. It is easy to see that
$|N[v]‎\cap (D‎\setminus \{u\})‎‎|‎\geq \lceil‎‎\frac{‎\delta‎}{2}‎‎‎\rceil‎‎‎$,
for all $v\in V(G)$. Hence, $D‎\setminus \{u\}$ is a
$\lceil‎‎\frac{‎\delta‎}{2}‎‎‎\rceil‎‎‎$-tuple dominating set in $G$. Hence,
$‎\gamma‎_{‎\times‎(‎\lceil‎‎\frac{‎\delta‎}{2}‎\rceil+1‎‎)}(G)‎\geq ‎\gamma‎_{‎\times‎‎\lceil‎‎\frac{‎\delta‎}{2}‎\rceil‎‎}(G)‎+1$.
By repeating this process, we obtain
$$\gamma‎_{‎\times‎(‎\lceil‎‎\frac{‎\delta‎}{2}‎\rceil+1‎‎)}(G)‎\geq
‎\gamma‎_{‎\times‎‎\lceil‎‎\frac{‎\delta‎}{2}‎\rceil‎‎}(G)‎+1‎\geq \ldots‎\geq
‎\gamma‎_{‎\times‎1}(G)+\lceil‎‎\frac{‎\delta‎}{2}‎\rceil=‎\gamma(G)‎‎‎‎‎+\lceil‎‎\frac{‎\delta‎}{2}‎\rceil.$$
By Inequality (\ref{EQ10}),
$$(n+\gamma_{s}(G))/2‎\geq ‎\gamma(G)‎‎‎‎‎+\lceil‎‎\frac{‎\delta‎}{2}‎\rceil.$$
This completes the proof of Claim 2.

\noindent The result now follows by Claim 1 and Claim 2. For sharpness consider the complete graph $K‎_{n}‎$.
\end{proof}
We conclude this section by establishing a lower bound on the signed domination number of a tree.
Dunbar et al. \cite{dhhs} proved that for every tree of order $n‎\geq2‎$,
$$\gamma_{s}(T)‎‎\geq ‎\dfrac{n+4}{3}‎‎.$$
Moreover, they showed that this bound is sharp.\\
We now present a
lower bound on $\gamma_{s}(T)‎$ of a tree $T$ of order $n‎\geq 2‎$ and
show that this bound is tighter than $(n+4)/3$.

\begin{theorem}\label{LBforTrees}
Let $T$ be a tree of order $n‎\geq2‎$ with $\ell$ leaves and $s$ support vertices. Then
$$
\gamma_{s}(T)‎\geq ‎\frac{(2‎\lceil‎‎\frac{‎\delta‎‎^{*}‎}{2}‎\rceil-1‎‎)n+2(\ell-s+2)}{2‎\lceil‎‎\frac{‎\delta‎‎^{*}‎}{2}‎\rceil+1}‎
$$
and this bound is sharp.
\end{theorem}

\begin{proof}
Let $f:V\rightarrow \{-1,1\}$ be a SDF of $T$ with weight $f(V(T))=\gamma_{s}(T)$.
If $V‎^{-}‎=‎\emptyset‎$, then $\gamma_{s}(T)‎=n$ and the result follows.
Suppose that $V‎^{-}‎\neq \emptyset‎$, and $u‎\in V‎^{-}‎$.
Root the tree $T$ at vertex $u$. For each vertex $v‎\in V‎^{-}‎‎$, let $P‎_{v}‎$ denote the set
of vertices $w$ satisfying $(i)$ $w$ belongs to $V‎^{+}‎$, $(ii)$ $w$ is a descendent of $v$,
and $(iii)$ each vertex of the $v$-$w$ path of $T$, except $v$, is in $V‎^{+}‎$.
Then the sets $P‎_{v}‎$, $v‎\in V‎^{-}‎‎$, partition the set $V‎^{+}‎$.

Let $S$ be the set of support vertices. We define
$$W‎_{0}‎=\{v‎‎\in V‎^{-}‎\setminus‎\{u\}|P‎_{v}‎\cap S=‎\emptyset‎‎‎\}$$
and
$$W‎_{1}‎=\{v‎‎\in V‎^{-}‎\setminus‎\{u\}|P‎_{v}‎\cap S‎\neq‎‎\emptyset\}.$$

\noindent Since $f(N[u])‎\geq1‎$, there are at
least $‎\lceil‎\frac{‎\delta‎‎^{*}‎}{2}‎‎\rceil‎‎+1$ children of $u$ that belong to $V‎^{+}‎$.
Moreover, each child of $u$ has at least one
child in $V‎^{+}‎$, itself. Therefore
\begin{equation}\label{EQ9.1}
|P‎_{u}‎|‎\geq 2(\lceil‎\frac{‎\delta‎‎^{*}‎}{2}‎‎\rceil‎‎+1)-|S‎\cap P‎_{u}‎‎|‎+|L‎\cap P‎_{u}‎‎|,
\end{equation}
where $L$ is the set of leaves in $T$.

\noindent Clearly, $V‎^{-}‎\setminus‎\{u\}=W‎_{0}‎\cup W‎_{1}‎$.
Every vertex $v$ in $‎V‎^{-}‎\setminus‎\{u\}$ has at least $‎\lceil‎\frac{‎\delta‎‎^{*}‎}{2}‎‎\rceil$
children in $V‎^{+}‎$ and each child has at least one child in $V‎^{+}‎$, itself. Hence,
\begin{equation}\label{EQ10.1}
|P‎_{v}‎|‎\geq 2‎\lceil‎\frac{‎\delta‎‎^{*}‎}{2}‎‎\rceil.
\end{equation}
Now let $v‎\in W‎_{1}‎‎$. Note that each support vertex and all leaves adjacent to it
belong to only one $P‎_{v}‎$, necessarily. Also in this process we have counted
just one leaf for every support vertex. This implies that

\begin{equation}\label{EQ11}
\sum_{v\in W‎_{1}‎}|P‎_{v}‎|‎\geq2‎\lceil‎\frac{‎\delta‎‎^{*}‎}{2}‎‎\rceil‎‎|W‎_{1}|-|S‎\cap ‎\cup‎_{v‎\in W‎_{1}‎‎}P‎_{v}‎‎‎‎|+|L‎\cap ‎\cup‎_{v‎\in W‎_{1}‎‎}P‎_{v}‎‎‎‎|.‎
\end{equation}

\noindent Together inequalities (\ref{EQ9.1}), (\ref{EQ10.1}) and (\ref{EQ11}) lead to

$$
\begin{array}{lcl}
|V‎^{+}‎|&=&|P‎_{u}‎|+\sum_{v\in W‎_{0}‎}|P‎_{v}|+\sum_{v\in W‎_{1}‎}|P‎_{v}|\\
&‎\geq ‎&2(‎\lceil‎\dfrac{‎\delta‎‎^{*}‎}{2}‎‎\rceil+1‎‎)+2‎\lceil‎\dfrac{‎\delta‎‎^{*}‎}{2}‎‎\rceil |W‎_{0}|+2‎\lceil‎\dfrac{‎\delta‎‎^{*}‎}{2}‎‎\rceil|W‎_{1}|+\ell-s.
\end{array}
$$
Using $|V‎^{-}‎\setminus \{u\}‎‎|=|W‎_{0}‎|+|W‎_{1}‎|$ we deduce that
$$
|V‎^{+}‎|‎\geq2(‎\lceil‎\dfrac{‎\delta‎‎^{*}‎}{2}‎‎\rceil+1‎‎)+2‎\lceil‎\dfrac{‎\delta‎‎^{*}‎}{2}‎‎\rceil(|V‎^{-}‎|-1)+(\ell-s).
$$
Now by the facts that $|V^{+}|=\dfrac{n+\gamma_{s}(G)}{2}$ and $|V^{-}|=\dfrac{n-\gamma_{s}(G)}{2}$
we obtain the desired lower bound.
\end{proof}

Since
$$
\frac{(2‎\lceil‎‎\frac{‎\delta‎‎^{*}‎}{2}‎\rceil-1‎‎)n+2(\ell-s+2)}{2‎\lceil‎‎\frac{‎\delta‎‎^{*}‎}{2}‎\rceil+1}‎‎\geq
‎‎\frac{n+4+2(\ell-s)}{3}‎,‎
$$
we conclude the following lower bound as an immediate result.
\begin{corollary}
Let $T$ be a tree of order $n‎\geq2‎$, with $\ell$ leaves and $s$ support vertices.
Then $\gamma_{s}(T)‎‎\geq ‎\dfrac{n+4+2(\ell-s)}{3}‎‎$.
\end{corollary}


\end{document}